\documentclass[11pt]{amsart}
\pdfoutput=1

\usepackage{amsmath, amssymb, amsthm, amsbsy, bbm, amsfonts, enumitem, color}
\usepackage{faktor}
\usepackage[margin=1.3in]{geometry}

\usepackage[all]{xy}
\usepackage[english]{babel}
\usepackage{graphicx, color, overpic, xcolor}

\usepackage[pdfpagelabels]{hyperref}

\DeclareGraphicsExtensions{.pdf, .jpg}

\theoremstyle{plain}
\newtheorem{thm}{Theorem}[section]
\newtheorem*{thm*}{Theorem}
\newtheorem{prop}[thm]{Proposition}
\newtheorem{lemma}[thm]{Lemma}
\newtheorem{corollary}[thm]{Corollary}

\theoremstyle{definition}
\newtheorem{definition}[thm]{Definition}
\newtheorem{example}[thm]{Example}

\newtheorem{notation}[thm]{Notation}

\theoremstyle{remark}
\newtheorem{remark}[thm]{Remark}

\newcommand{\R}{\mathbb{R}}
\newcommand{\Z}{\mathbb{Z}}

\newcommand{\C}{\mathbb{C}}

\newcommand{\define}{\mathrel{\mathop:}=}

\newcommand{\RS}{\Phi}
\newcommand{\App}{\mathcal{A}}
\newcommand{\bound}{\partial} 
\newcommand{\sW}{W_0} 
\newcommand{\aW}{W} 
\newcommand{\Cf}{{\mathcal{{C}}_{f}}} 
\newcommand{\fa}{{\bf{c}_f}} 
\newcommand{\Cfm}{{\mathcal{{C}}_{w_0}}} 
\newcommand{\stab}{\mathrm{Stab}} 
\newcommand{\geomSigma}{\vert\Sigma\vert} 

\numberwithin{equation}{thm}

\begin{document}

\hypersetup{pdfauthor={Schwer},pdftitle={Root operators, root groups and retractions}}.
\title{Root operators, root groups and retractions}
\author{Petra Schwer}
\address{Petra Schwer, Department of Mathematics, Karlsruhe Institute of Technology, Kaiserstr. 89-93, 76133 Karlsruhe, Germany}
\email{petra.schwer@kit.edu}

\thanks{The author was partially supported by the DFG Project SCHW 1550/2-1.}
\date{ \today }

\begin{abstract}
We prove that the Gaussent--Littelmann root operators on galleries can be
expressed purely in terms of retractions of a (Bruhat-Tits) building. In addition we  establish a connection to the root datum at infinity.  
\end{abstract}

\maketitle

\section{Introduction}

Root operators were introduced by Littelmann \cite{LittelmannPaths} in the context of the path model for finite dimensional representations of a connected complex semisimple algrebraic group $G$. 
They provide a method to modify paths in an apartment in a controlled way determined by a chosen, fixed, root direction.  
Later Gaussent and Littelmann \cite{GaussentLittelmann} defined a version for galleries while establishing a connection between the path model and the geometry of the affine Grassmanian. 

It is mentioned in some places in the literature that there is a connection between root operators and retractions in the affine building. In this note we will make this connection explicit and prove, in Theorem~\ref{thm:ef} and \ref{thm:etilde},  that the root operators can be expressed purely in terms of the buildings' retractions. In addition, we will give an interpretation using root groups of the spherical building at infinity. 

The proof of our main statement does not depend on the cardinality of the branching of the building. The only assumption we need is, that the building in question is thick. In particular, we do not need to assume that the building is of Bruhat-Tits type. And, if it comes from a group, then the proof does not depend on the underlying field. The heuristic reason is that a root operator maps a gallery which is contained in an apartment to another gallery in the same apartment. Hence they don't  ``see'' the branching of the building. 

Retractions are an essential tool in the theory of buildings which appear in various applications.  In case that the building is the Bruhat-Tits building of some algebraic group they are strongly linked to two kinds of decompositions of the group: the Iwasawa and Cartan decomposition. We make this connection explicit in Section~\ref{sec-retractions}.  

The (pre-)images of the retractions in question are strongly connected to certain kinds of double coset intersections of subgroups of a semisimple algebraic group. Numerous works make use of this connection, such as \cite{Hitzelberger}, \cite{ParkinsonRamSchwer}, \cite{GaussentLittelmann}, \cite{KapovichMillson}, \cite{KLM1} or \cite{KLM3}, to name just a few. 
 Mili\'{c}evi\'{c}, Thomas and the author recently used this connection (as well as the root operators themselves) in their new approach to affine Deligne-Lusztig varieties in  \cite{ADLV}.

We start by recalling the definition of buildings, their root groups and the Bruhat and Iwasawa decomposition  in Section~\ref{sec-prelim}. In Section~\ref{sec-retractions}, we introduce a retraction in an affine building and give its interpretation in terms of subgroups of $G$. We include the definitions of the root operators in Section~\ref{sec-operators}. The main results of the paper are then proven in Section~\ref{sec-main}. 

As the anonymous referee kindly pointed out, the results in this paper should also hold for mazures, which are the analogs of affine buildings in the Kac-Moody setting. I have not carried out the generalization at this point, but agree that this seems doable and should be done.   

\subsection*{Acknowledgements} I would like to thank Jacinta Torres for helpful comments on an earlier draft. 


\section{preliminaries}\label{sec-prelim}

In this section we quickly recall the definition and some basic properties of buildings. In particular root groups and retractions are introduced here. For more details please refer to standard textbooks such as \cite{Brown} or \cite{Ronan}. 

\subsection{Buildings and root systems}\label{sec:root systems}

We start with a formal definition of buildings as (abstract) simplicial complexes. 

\begin{definition}
A \emph{building} $X$ is a simplicial complex which is the union of subcomplexes, called \emph{apartments}, such that each apartment is isomorphic to some (geometric realization of) a Coxeter complex $\Sigma$ and the following two axioms are satisfied:
\begin{itemize}
\item[(B1)] For any two simplices there exists an apartment containing both.
\item[(B2)] If A and A' are two apartments containing simplices $\sigma$ and $\tau$, then there exists an isomorphism $A\to A'$ fixing their intersection pointwise.
\end{itemize}
The set of all apartments is called an \emph{atlas} $\App$ of $X$.  
\end{definition}

It is easy to see that for any building all apartments in $\App$ are pairwise isomorphic and hence the type of the associated Coxeter group $W$ acting on the apartment is the same for all $A\in\App$. We refer to this as the type of the building.  Buildings are chamber complexes, i.e. connected simplicial complexes, where the maximal simplices are all of the same, finite dimension. 
A building is \emph{thick} if each  co-dimension one face of a maximal simplex is the face of  at least three chambers. Thickness is not a strong condition as every (potentially non-thick) building has a canonical thickening (compare \cite{Scharlau} and  \cite[Sec 3.7]{KleinerLeeb}) and all buildings of higher rank are automatically thick.  

Affine Coxeter groups $\aW$ admit a canonical splitting as a semi-direct product of the associated spherical Weyl group $\sW$ with the co-root lattice $R^\vee$ of the underlying root system $\RS$, compare for example \cite{Bourbaki4-6}. 

For a fixed apartment $A$ in a building $X$ of type $\aW=\sW\ltimes \RS^\vee$ we identify $A$ with the $\R$-span of the simple co-roots in $\RS^\vee$.
We enumerate the simple roots by  $\alpha_1, \ldots, \alpha_n$ and denote the corresponding co-roots by $\alpha_1^\vee, \ldots, \alpha_n^\vee$. 

The group $\aW$ is a reflection group of the geometric realization $\geomSigma$ of its affine Coxeter complex equipped  with the usual euclidean metric. The subgroup $\sW$ can be identified with the stabilizer of the origin $v_0$ and the images of $v_0$ under the translations $t_\mu$, with $\mu \in \RS^\vee$, are the vertices of the same type as $v_0$.  

For every pair of a root $\alpha$ and an integer $m\in\Z$, there is a wall $H_{\alpha,m}=\{x\in A\;\vert \langle x,\alpha^\vee\rangle=m\}$ in $A$, which determines a positive (with respect to $\alpha$)  half-space $H_{\alpha,m}^+=\{x\in A\;\vert \langle x,\alpha^\vee\rangle \geq m\}$ and a negative half-space $H_{\alpha,m}^-=\{x\in A\;\vert \langle x,\alpha^\vee\rangle \leq m\}$.  If we restrict to positive roots the set of hyperplanes is in bijection with with the pairs of roots and elements of $\Z$, that is with $\RS^+\times\Z$. 

The closures of the connected components of $A\setminus \bigcup_{m\in\Z, \alpha\in\RS^+}  H_{\alpha, m}$ are called \emph{alcoves}  and are in one-to-one correspondence with the elements of the affine Weyl group $\aW$.
 They are the maximal simplices in the simplicial structure of $\Sigma$. The $0$-simplices will be called \emph{vertices} and the codimension one simplices will be referred to as  \emph{panels}.    
 The \emph{fundamental alcove} in $\geomSigma$ is the set 
 \[
 \fa=\{x\in \geomSigma\;\vert\; 0\leq \langle x,\alpha\rangle\leq 1,  \;\forall  \alpha\in\RS^{+}\}
 \]
and corresponds to the identity. Here we choose as a generating set of $\aW$ the set of reflections $\tilde S=\{s_0, s_1, \ldots, s_n\}$, where $s_i$ is the reflection on the wall $H_{\alpha_i, 0}$ for all $i\neq 0$ and where we put $s_0$ to be the reflection on the wall of index one perpendicular to the highest root $\alpha$ in $\RS$. With this setup the set $S=\{s_1, \ldots, s_n\}$ generates the associated spherical Weyl group $\sW$. 

We label the panels in $\Sigma$ with the generators $\{s_0, s_1, \ldots, s_n\}$ in such a way that the labeling is invariant under the $\aW$-action and such that panel of the fundamental alcove which is fixed by $s_i$ has label $s_i$. 
 
As we can identify all apartments $A$ with $\Sigma$ we may speak about alcoves, hyperplanes and labeled panels, etc in all apartments of the building.  

The chambers of the spherical building $\bound X$ at infinity of $X$ are the parallel classes of Weyl chambers in $X$. The map that sends an apartment $A$ of $X$ to the union $\bound A$ of all parallel classes of Weyl chambers in $A$ is a bijection between the set of apartments in $X$ and the set of apartments in the spherical building $\bound X$ at infinity. 
To more easily distinguish them from the alcoves in $X$ we call the maximal simplices in $\bound X$  \emph{chambers}. 

\subsection{Root groups}

The root group datum of a  (semisimple) algebraic group $G$ together with a valuation of this datum fully determines the associated Bruhat-Tits building $X$. Here a crucial role is played by the root groups of $\partial X$. Moreover $G$ admits Cartan, Bruhat  and Iwasawa decompositions, which can be stated in geometric terms as properties of the associated building. 
We now recall these decompositions and  introduce roots groups and valuations of root data. For further details refer the reader to Section 11 of \cite{AB}, Section 6,7 and in particular 7.3 of \cite{BruhatTits} or Chapters 3 and 13  of \cite{AffineWeiss}. 
We suppose that the buildings considered are irreducible. 

Let in the following $G$ be an algebraic group over a field $F$ with a discrete valuation $\nu$. Suppose $G$ has an affine Tits system (or BN-pair) $(B,N)$ with an associated irreducible Bruhat-Tits building $X$. We fix an apartment $A$ in $X$ together with an origin $v_0\in A$ and write $\bound A$ for the spherical apartment that is  the boundary of $A$ at infinity. We will write $\fa$ for the fundamental alcove in $A$ (which contains $v_0$ as a vertex) and denote by $\Cf$ be the fundamental Weyl chamber, that is the unique Weyl chamber based at $v_0$ containing the fundamental alcove $\fa$. 

The group $N$ of the Tits system then is the stabilizer in $G$ of the apartment $A$ and the group $B$ is the subgroup of $G$ stabilizing the fundamental alcove $\fa$. Note that $B$ is often denoted by $I$ and referred to as the Iwahori subgroup of $G$. The spherical Weyl group $\sW$ equals the stabilizer $\stab_\aW(v_0)$ of the origin in the affine Weyl group $\aW$, while $K$ is the stabilizer of $v_0$ in $G$. In case that $F$ is locally compact the group $K$ is a maximal compact subgroup of $G$. We write $T$ for the sub-group of translations in $\aW$ and $T_\Cf$ for the translations $t$ in $\aW$ with  $t v_0\in\Cf$ and $U$ for the stabilizer in $G$ of the chamber $\bound \Cf$ at infinity.

\begin{prop}\label{prop:decomp}
With $\Cf, \fa$ and $v_0$ as above and  subgroups $K=\stab_G(v_0)$, $U=\stab_G(\bound \Cf)$ and $B=\stab_G(\fa)$ in $G$, the group $G$ decoposes as follows:  
\begin{enumerate}
\item\label{Bruhat} Bruhat decomposition \[G= \bigsqcup_{w\in\aW} BwB.\]   
\item\label{Iwasawa} Iwasawa decomposition \[G= \bigsqcup_{t\in T} UtK,\] 
\item\label{Cartan} Cartan decomposition \[G= \bigsqcup_{t\in T_\Cf}Kt K, \] 
\end{enumerate}
\end{prop}
For proofs see (4.4.3) of \cite{BruhatTits} and Lemma 5.1 in \cite{Ronan}.

With notation as above the decompositions in Proposition~\ref{prop:decomp} translate to the following geometric statements. 

\begin{prop}\label{prop:union}
Let $X$ be a (thick) affine building, $c$ an alcove in $X$ and $\bound C$ a chamber in $\bound X$. Then 
\begin{enumerate}
\item $X$ is the union of all apartments containing $c$. 
\item\label{} $X$ is the union of all apartments $A$ such that $\bound \Cf\subset \bound A$. 
\end{enumerate}
\end{prop}

In the Bruhat-Tits case the first item can be deduced from the Bruhat decomposition and the second is a consequence of the Iwasawa decomposition. However, it is not hard to see that one can prove these two statements directly from the definition of affine buildings and their spherical buildings at infinity.

%

\begin{definition} [(products of) root groups] 
Let in the following $\alpha$ be a half-apartment
of $\bound A$ in the spherical building $\bound X$ at infinity of $X$. Define the \emph{root group}  $U_\alpha$ of $\alpha$  to be the subgroup of the full automorphism group of $\bound X$ that fixes every chamber having a panel contained in $\alpha\smallsetminus\bound\alpha$.

For a Weyl chamber $C$ in $X$ let $U_C$ be the product of all root groups $U_\alpha$, where $\alpha$ contains $\bound C$.
\end{definition}

Thus the group $U$, introduced right before Proposition~\ref{prop:decomp}, satisfies 
\[
U=U_\Cf=\prod_{\bound C\in\alpha} U_\alpha.
\]  

One can define two odered sets of roots, the \emph{closed}  \emph{interval} $[a,b]=(\alpha_0, \alpha_1, \ldots, \alpha_s)$ and the \emph{open} interval $(\alpha,\beta)=( \alpha_1, \ldots, \alpha_{s-1})$) of roots $\alpha$ and $\beta$ which are not opposite each other. For a precise definition see \cite[Defintion 3.1]{AffineWeiss}. 
As a set the closed interval  $[\alpha, \beta]$ consists of all roots of $\bound A$ containing $\alpha\cap\beta$.  

\begin{prop}\cite[Prop. 3.2]{AffineWeiss}
With notation as above let $[\alpha, \beta]$ be a closed interval of non-opposite roots. Then 
\begin{enumerate}
\item if $s\geq 3$ one has $[U_1,U_s]\subset U_2 U_3\cdots U_{s-1}$, where $U_i$ denotes the root group $U_{\alpha_i}$ for all $i\in\{1,2,\ldots, s\}$. 
\item If $s=2$ one has $[U_1,U_s]=1$. 
\item Every element of $\langle U_1, U_2, \ldots, U_s\rangle$ can be written uniquely as a product of the form $u_1u_2\cdots u_2$ with $u_i\in U_i$ for all $i\in\{1,2,\ldots, s\}$. 
\end{enumerate}
\end{prop}

There is a natural one-to-one correspondence between the half-apartments of $\bound A$ and the elements of the underlying root system $\RS$.  In the following we will once and for all fix such an identification and will index the set of half-apartments in $\bound A$ by elements of $\RS$ and call both of them \emph{roots}. We refer the reader to Definition 3.12 and Proposition 3.14 of \cite{AffineWeiss} where this identification is made precise. 
Compare also Remark 13.16 of \cite{AffineWeiss}.

\begin{definition}
Let $G^\dagger$ be the group generated by all the root groups $U_\alpha$  for all roots $\alpha$ in $\bound A$. Let $\xi$ denote the map from $\RS$ to $G^\dagger$ that assigns to a root $\alpha$ the root group $U_\alpha$. We call the triple  $(G^\dagger, \{U_\alpha\}_{\alpha\in\RS}, \xi)$ \emph{root datum} of $\bound X$ \emph{based at} $\bound A$. 
\end{definition}

In case that $F = \C((t))$ where $\C$ are the complex numbers, $G^\dagger$ is just G. 
It is shown in \cite[3.4 and 29.15.iii]{AffineWeiss} that $G^\dagger$ acts transitively on the set of apartments of $\bound X$ and hence the root datum is unique up to conjugation by an element of $G^\dagger$.  

%

Recall that we had fixed an apartment $A\subset X$ and an origin $v_0\in A$ and the roots $\alpha\in\RS$ are in one to one correspondence with the half-apartments in $A$ determined by walls through the origin. 
Recall that the set of all half-apartments in $A$ is $\left\{ H^\pm_{\alpha, k} \;\vert\;  k\in\Z, \alpha\in\RS^+\right\}$. 
     
\begin{prop}\label{prop:fixed points}
Let $\alpha$ be a root and $u\in U_\alpha^\ast$. Then the fixed point set $a_u\define A\cap A^u$ of $u$ in $A$ is a half-apartment of $A$ with $\bound a_u=\alpha$.  In particular elements of $U_a^\ast$ are special automorphisms of $\bound A$ and for each of them exists $k\in\Z$ such that $a_u=H^+_{\alpha, k}$, the positive half-apartment of index $k$ with respect to $\alpha$.  
\end{prop}
\proof 
This is a consequence of Propositions 13.2 and 13.18.ii and Notation 13.17 of \cite{AffineWeiss}.
\qed

\begin{definition}[{\cite[3.21]{AffineWeiss}}]
A \emph{valuation of the root datum} $(G^\dagger, \{U_\alpha\}_{\alpha\in\RS}, \xi)$ is a  family of maps $\RS_\alpha:U^\ast_\alpha \to \Z$, where we put for all $\alpha\in\RS$ and all $u\in U_\alpha^\ast$ the map $\RS_\alpha(u)=-k$ with  $k$ as in \ref{prop:fixed points}.
 We let $\RS_\alpha(1)=\infty$ for all $\alpha$ to extend $\RS_\alpha$ to all of $U_\alpha$. 
\end{definition}

Note that $\RS_\alpha$ is in fact surjective for all $\alpha$. 

\begin{definition}[Affine root groups]
For $\alpha\in\RS$ we set 
\[
U_{\alpha, k}\define\{u\in U_\alpha \;\vert\; \RS(a)\geq k \}  
\]
for each $k\in\R$ (or in $\Z$). 
\end{definition}

\begin{remark}
One can show \cite[13.18.iii]{AffineWeiss} that $u\in U_{\alpha, k}^\ast$ if and only if $H^+_{\alpha, -k}\subset a_u$.  
Thus the group  $U_{\alpha, k}^\ast$ can be written as 
\[
U_{\alpha, k}^\ast=\{ u\in U_\alpha \;\vert \; a_u=H^+_{\alpha, l} \text{ with } l\leq - k\}
\]
Moreover, the elements in $U_{\alpha, k}^\ast \setminus U_{\alpha, k-1}^\ast$ are the ones with $a_u=H^+_{\alpha, -k}$ and $\phi_\alpha(u)=k$. 
\end{remark}

\begin{definition}\label{def:m(u)}
Let $\alpha$ be a root. Then, see 3.8 of \cite{AffineWeiss}, for each $u\in U_\alpha^\ast$ there exists a unique element in $U^\ast_{-\alpha}uU^\ast_{-\alpha}$ that maps $\alpha$ to $-\alpha$. We call this element $m(u)$. 
\end{definition}

One can show that 
\[ m(u)^{-1} = m(u^{-1}) \;\text{ for all } u \in U_\alpha^\ast.\] Therefore every $g\in m(U_\alpha^\ast)$ induces the reflection $s_\alpha$ on $\bound A$ which interchanges $\alpha$ and $-\alpha$. 
For elements $u \in U_{\alpha, k}^\ast \setminus U_{\alpha, k-1}^\ast$ we get that $m(u)$ induces the reflection $s_{\alpha, -k}$ on $A$ which switches the half spaces $H^+_{\alpha, -k}$ and $H^-_{\alpha, -k}$.


\subsection{Retractions}\label{sec-retractions}

In this subsection we will recall the definition of the retraction ``from infinity'' in an affine building onto a fixed apartment and explain their connection to the Bruhat and Iwasawa decomposition. 

Proposition~\ref{prop:union} allows us to define a retraction of $X$ onto a fixed apartment which depends on the choice of a chamber in $\bound A$.  There is a second type of retraction given by a choice of an alcove in $A$ which we will not need in this paper and hence won't introduce. 

\begin{definition}\label{retraction infty}
Let $\bound C$ be a chamber in a fixed apartment $\bound A$ in $\bound X$.  For every alcove $d$ in $X$ choose $A'\in\App$ such that $d\in A'$ and $\bound C\subset \bound A'$ and define 
$$
\rho_{A, C}(d)\define \varphi(d),
$$ 
where $\varphi:A'\to A$ is the unique isomorphism mapping $A'$ to $A$ and fixing their intersection. We call $\rho_{A,  C}$ the \emph{retraction (from infinity) onto $A$ based at $\bound C$}. 
\end{definition}

From Proposition~\ref{prop:union} and the definition of buildings one easily deduces that these retractions are well defined and independent of the choice of an apartment $A'$. 

The retractions $\rho_{A, C}$ have a natural group theoretic interpretation, which we will state and prove below. 
Compare also Proposition 1 of \cite{GaussentLittelmann}. 
   
\begin{prop}
The fibers of $\rho_{A,  C}: X\to A$ are the $U_C$ orbits on $X$.  
\end{prop}
\begin{proof}
Recall that for an arbitrary Weyl chamber $C$ the group  $U_C$ is the product of all root groups containing $\bound C$ at infinity. 
The action of $U_C$ on the set of all apartments containing a sub-Weyl chamber of $C$ is transitive. Moreover for each $u\in U_C$ the restriction of $u$ to $uA$ is the unique isomorphism mapping $uA$ to $A$ and fixing their intersection pointwise. Therefore for all points $x'$ in the preimage $\rho_{A, C}^{-1}(x)$,  for some $x\in A$, there exists $u\in U_C$ with $x'=ux$. We thus have the assertion.   
\end{proof}


\section{Definition of root  operators}\label{sec-operators}

We now recall the definition of the operators $e_\alpha$, $f_\alpha$ and $\tilde e_\alpha$ for simple roots $\alpha$ in a fixed apartment $A$ as introduced in \cite{GaussentLittelmann}. They take as an input a combinatorial gallery. 

\begin{definition}
A \emph{combinatorial gallery} $\gamma$ is a sequence of simplices in $X$, 
\[
\gamma=\left(p_0\subset c_0\supset p_1\subset \cdots c_l\supset p_{l+1} \right),
\] 
where $p_0, p_{l+1}$ are vertices in $X$, the $c_i$ are simplices of dimension $\geq 1$ and the $p_i$, $i\neq 0,l+1$, are faces of both $c_{i-1}$ and $c_i$ of positive codimension.    

Every panel $p_i$ in $\gamma$ is labeled by some $s_{j_i}\in S$ and we call the product $w=s_{j_1}s_{j_2}\cdots s_{j_n}$ the type of the gallery. 
\end{definition}

A combinatorial gallery where all $c_i$ are alcoves is an ordinary gallery of alcoves where in addition a start and end vertex as well as for each pair of consecutive alcoves a shared codimension one face is specified. 


\begin{notation}\label{not:operators}
Let $\gamma=\left(p_0\subset c_0\supset \cdots c_l\supset p_{l+1} \right)$ 
be a combinatorial gallery of type $\gamma_\lambda$ that starts in the origin and ends in a co-character $\nu\preceq\lambda$, i.e. $\gamma\in\Gamma(\gamma_\lambda, \nu)$.   Let $\alpha$ be a simple root, and define $m\in\Z$ to be minimal such that there exists $q$ with $p_q$ contained in the hyperplane $H_{\alpha, m}$. Note that $m\leq 0$ as $p_0$ is the origin. 
 \noindent
There are the following cases:
\begin{itemize}
\item[(I)] $m\leq -1$. In this case let  $k$ be minimal with $p_k\subset H_{\alpha, m}$, and let $0\leq j\leq k$ be maximal with $p_j\subset H_{\alpha, m+1}$. 
\item[(II)] $m\leq \langle \nu, \alpha \rangle - 1$. In this case let $j$ be maximal with $p_j\subset H_{\alpha, m}$, and let $j\leq k\leq l+1$ be minimal  with $p_k\subset H_{\alpha, m+1}$.
\item[(III)] $\gamma$ crosses $H_{\alpha, m}$. In this case fix $j$ minimal such that $p_j\subset H_{\alpha, m}$ and $H_{\alpha, m}$ separates $c_i$ from $\Cf$ for all $i<j$. Let $k>j$ be maximal such that $p_k\subset H_{\alpha, m}$.      
\end{itemize}
\end{notation}
\noindent
Observe that cases (I) -- (III) are not disjoint.

\begin{definition}\label{def:GLoperators}
With notation as in \ref{not:operators}, we define \emph{root operators} $e_\alpha$,  $f_\alpha$ and $\tilde e_\alpha$ as follows:
\begin{itemize}
\item In case (I) let $e_\alpha(\gamma)$ be the combinatorial gallery defined by
\[ e_\alpha(\gamma)=(\nu=p_0\subset c'_0\supset p'_1 \subset c'_1 \supset p'_1 \subset \ldots \subset c'_l \supset p'_{l+1}=\lambda ), \]
where 
\[ c'_i = 
\left\{ \begin{array}{ll}
	  c_i & \text{for } i< j-1, \\
	  s_{\alpha, m+1}(c_i) & \text{for } j\leq i< k, \\
	  t_{\alpha^\vee}(c_i) & \text{for } i\geq k. \\
        \end{array}
\right .
\]

\item  
In case (II) let $f_\alpha(\gamma)$ be the combinatorial gallery defined by
\[f_\alpha(\gamma)=(\nu=p_0\subset c'_0\supset p'_1 \subset c'_1 \supset p'_1 \subset \ldots \subset c'_l \supset p'_{l+1}=\lambda ), \]
where 
\[ c'_i = 
\left\{ \begin{array}{ll}
	  c_i & \text{for } i< j, \\
	  s_{\alpha, m+1}(c_i) & \text{for } j\leq i< k, \\
	  t_{-\alpha^\vee}(c_i) & \text{for } i\geq k. \\
        \end{array}
\right .
\]

\item In case (III) let $\tilde e_\alpha$ be the combinatorial gallery defined by 
\[\tilde e_\alpha(\gamma)=(\nu=p_0\subset c'_0\supset p'_1 \subset c'_1 \supset p'_1 \subset \ldots \subset c'_l \supset p'_{l+1}=\lambda ), \]
where 
\[ c'_i = 
\left\{ \begin{array}{ll}
	  c_i & \text{for } i\leq j-1 \text{ and } i\geq k, \\
	  s_{\alpha, m+1}(c_i) & \text{for } j\leq i< k.
	  \end{array}
\right .
\]
 
\end{itemize} 
\end{definition}

The $e$ and $f$ operators are partial inverses of one another and have many nice combinatorial properties listed in Lemma 6 and 7 of \cite{GaussentLittelmann}. 

\section{Expressing root operators in terms of retractions and root group elements} \label{sec-main}

In this section we show that one can express root operators in terms of retractions and with the help of root group elements.  We comment on the path--version of the root operators in Subsection~\ref{sec:paths}.   

\subsection{Expressing root operators in terms of retractions}

As always let $X$ be a thick affine building, $A$ a fixed apartment in $X$ and $\alpha$ a simple root. 
For every $k\in\Z$ choose an apartment $A_k\in\App$  such that 
\[
A\cap A_k=H_{\alpha,k}^-
\]
and write $\rho_k$ for the restriction of $\rho_{A,\Cfm}:X\to A$ to the apartment $A_k$. 
Note that $\bound\Cfm\subset\bound A_k$. It is easy to see that $\rho_k$ is an isometry that fixes the intersection $A_k\cap A$. 
Similarly write $\rho_{k}^{op}$ for the restriction of $\rho_{A,\Cf}$ to the apartment 
\[
B_{k}\define (A_{k}\setminus A) \cup (A\setminus A_{k}).
\] 

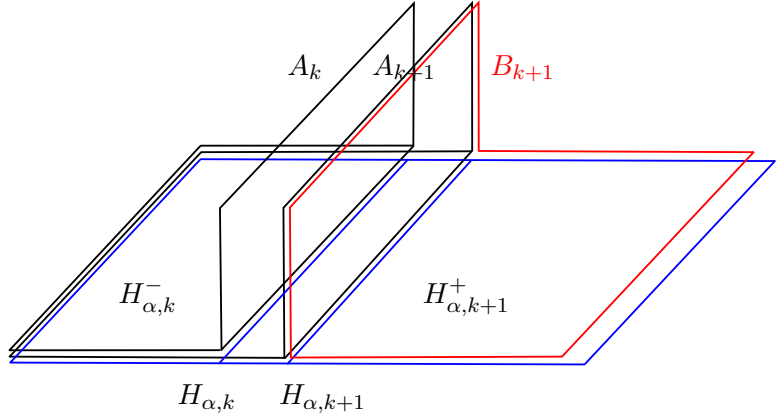
\begin{figure}[h]
\begin{overpic}[width=0.75\textwidth]{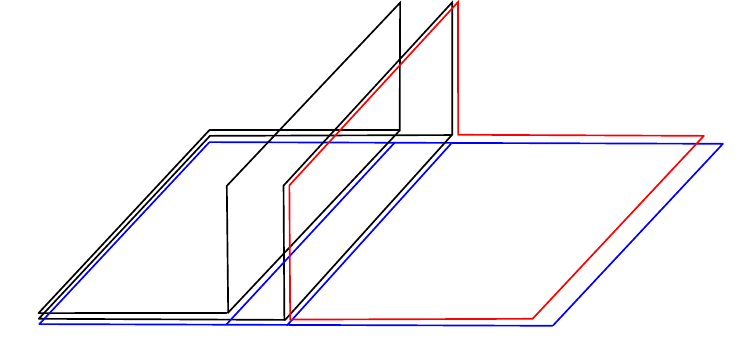}
\put(38,37){$A_{k} $}
\put(48,37){$A_{k+1} $}
\put(62,37){\color{red}$B_{k+1} $}
\put(25,-2){$H_{\alpha, k}$ }
\put(37,-2){$H_{\alpha, k+1}$ }
\put(18,10){$H^-_{\alpha, k}$ }
\put(54,10){$H^+_{\alpha, k+1}$ }
\end{overpic}
\caption[retractions]{The intersection of the apartments $A_k$ and $B_k$ with $A$ are  complementing half-spaces.}
\label{fig:half-spaces}
\end{figure}

We define maps between pairs of apartments $B_k$ and $B_l$.

\begin{definition}
With notation as above put
\[
\iota_{k,l, A}\define (\rho_k)^{-1}\circ\rho_l. 
 \text{\; and  let }
\iota_{k,l, B}\define (\rho^{op}_k)^{-1}\circ\rho^{op}_l 
\]
\end{definition}

\begin{lemma}
The map $\iota_{k,l, B}:B_l \to B_k$ is an isometry from $B_k$ to $B_l$  fixing $B_l \cap B_k$ pointwise. 
Similarly, the map $\iota_{k,l, A}: A_l \to A_k$ is an isometry from $A_l$ to $A_k$ fixing $A_l \cap A_k$ pointwise.
\end{lemma}
\begin{proof}
The restriction of $\rho_{A,\Cf}$, respectively $\rho_{A,\Cfm}$, to an apartment $A'$ containing a sub-Weyl chamber of $\Cf$, respectively $\Cfm$, is an isometry between $A'$ and $A$  that fixes their intersection pointwise. Hence the lemma. 
\end{proof}

The next lemma shows how reflections in an apartment are linked with retractions. 

\begin{lemma}\label{lem:retraction reflection}
For any alcove $c\in A_{k}\cap B_{k}$ its two retracted images $d=\rho_{k}(c)$ and $d^{op}=\rho_{k}^{op}(c)$ are reflected images of one another along $H_{\alpha, k}$.  That is $d^{op}=s_{\alpha,k}(d)=m(u)(d)$ for all $u\in U_\alpha^\star  $. 
\end{lemma}
\begin{proof}
Suppose $C$ is a Weyl chamber in $A$ and $B\in\App$ such that  $\bound B$ contains  $\bound C$. Then the restriction of  $\rho_{A,C}$ to $B$ is an isometry of apartments fixing $A\cap B$ pointwise. Therefore the half space $B_{k}\setminus A$ is  isometrically mapped onto $A\setminus B_{k}$ by $\rho_{k}^{op}$. The same statement holds for the half-space $A_{k}\setminus A$ and $A\setminus A_{k}$ and the retraction $\rho_{k}$. 
Combining the two we obtain an isometry $A\setminus A_{k}\to A\setminus B_{k}$ that fixes the common wall. This isometry has to be the reflection along the wall $H_{\alpha, k}$. 
\end{proof}


We introduce some notation and state our main result afterwards. 

\begin{notation}[Writing galleries as concatenations of sub-galleries]
Any combinatorial gallery $\gamma=(f_0\subset c_1\supset f_1\subset\ldots \subset c_n\supset f_n)$ can be split into two sub-galleries 
$$
\gamma_i^-=(f_0\subset c_1\supset \ldots \subset c_{i-1}\supset f_i)
$$
and 
$$
\gamma_i^+=(f_i\subset c_k\supset \ldots \subset c_n\supset f_n), 
$$
for all $0\leq i\leq n$. With this $\gamma=\gamma_i^-\star\gamma_i^+$, the concatenation of $\gamma_i^-$ and $\gamma_i^+$.  
\end{notation}

\begin{thm}\label{thm:ef}
Let $\Cfm$ be the unique Weyl chamber in $A$ opposite $\Cf$. Let $\gamma$  be a combinatorial gallery and let $k\geq 1$, $m\leq -1$ and $j$ be as in \ref{not:operators}. 

Suppose that $e_\alpha$ is defined for $\gamma$ and decompose $\gamma$ as $\gamma=\gamma_k^-\star\gamma_k^+$. Then 
$$
e_\alpha(\gamma)=\rho_{A, \Cfm}\left( \iota_{m+1, m, B}\left(\gamma_k^-\star\rho_m^{-1}(\gamma_k^+)\right)\right).
$$

Suppose $f_\alpha$ is defined for $\gamma$ and decompose $\gamma$ as $\gamma=\gamma_j^-\star\gamma_j^+$. Then 
$$
f_\alpha(\gamma)=\rho_{A, \Cfm}\left( \iota_{m-1, m, B}\left(\gamma_j^-\star\rho_m^{-1}(\gamma_j^+)\right)\right).
$$
\end{thm}

\begin{proof}
Retractions in buildings preserve adjacency and dimension of simplices, therefore 
\[
\tilde\gamma\define \rho_{A, \Cfm}( \iota_{m+1, m, B}(\gamma_k^-\star\rho_m^{-1}(\gamma_k^+)))
\]
 is again a combinatorial gallery. 
Let $c$ be one of the $c_i$ in $\gamma$ and let $\tilde c$ denote the image of $c$ in the gallery $\tilde\gamma$. In order to prove the statement we need to see that $\tilde c$  is the same as the image of $c$ under the operator $e_\alpha$. There are three cases: either $c$ is in $\gamma_k^+$, or $c$ is in $\gamma_k^-$ where we distinguish between the case that $c$ is in the strip between $H_{\alpha, m}$ and $H_{\alpha, m+1}$, or $c\in \gamma_k^-$ is in the half space $H_{\alpha, m+1}^+$. 

Suppose first that $c\in \gamma_k^- \cap H_{\alpha, m+1}^+$. In this case $c=c_j$ for some $j<k$ and the root operator $e_\alpha$ thus fixes $c$. As $c\in H_{\alpha, m+1}^+$ one can see that $\rho^{op}_{i}(c)=c$, for $i=m, m+1$ and $\rho_{A, \Cfm}(c)=c$. Hence
$$
\rho_{A, \Cfm}( \iota_{m+1, m, B}(c)) = \rho_{A, \Cfm}\left(\left( (\rho^{op}_{m+1})^{-1}\circ\rho^{op}_m \right)(c)\right) = \rho_{A, \Cfm}(c)=c.
$$

Suppose now that $c\in \gamma_k^-$ is in the strip between $H_{\alpha, m}$ and $H_{\alpha, m+1}$. In this case $e_\alpha$ reflects $c$ along $H_{\alpha, m+1}$. To see what the map on the right hand side of the equation does argue as follows: as  $\rho^{op}_m(c)=c$ we have that
$$
\rho_{A, \Cfm}\left( \iota_{m+1, m, B}(c)\right) = \rho_{A, \Cfm}\left( (\rho^{op}_{m+1})^{-1} (c)\right).
$$
Lemma~\ref{lem:retraction reflection} implies that in fact $\rho_{A, \Cfm}( \iota_{m+1, m, B}(c)) =c$. 

In the last case $c\in\gamma_k^+$. The image of $c$ by $e_\alpha$ is then the translate $t_{\alpha^\vee}(c)$. We need to verify that $\tilde c=t_{\alpha^\vee}(c)$. As $c\in\gamma_k^+$ we may conclude by  Lemma~\ref{lem:retraction reflection} that
$$
\tilde c 
= \rho_{A, \Cfm}\left( \left((\rho^{op}_{m+1})^{-1}\circ\rho^{op}_m\right)\left(\rho_m^{-1}(c)\right)\right)
= \rho_{A, \Cfm}\left( (\rho^{op}_{m+1})^{-1} \left( s_{\alpha, m} (c) \right) \right).
$$
The reflected image $s_{\alpha, m} (c) $ of $c$  is contained in $A\setminus B_{m+1}$ and thus the simplex $c'\define (\rho^{op}_{m+1})^{-1} \left( s_{\alpha, m} (c) \right)$ is contained in $B_{m+1}\setminus A$ which implies that $\rho_{A, \Cfm}(c')=\rho_{m+1}(c')$. Another application of Lemma~\ref{lem:retraction reflection} thus implies that 
$
\tilde c = s_{\alpha, m+1}\left( s_{\alpha, m} (c)  \right) = t_{\alpha^\vee}(c), 
$
which completes the proof in the first case. The the formula for $f_\alpha$ is obtained similarly.  
\end{proof}

Finally we study the operator $\tilde e_\alpha$. 
\begin{thm}\label{thm:etilde}
Suppose that $\gamma=(f_0\subset c_1\supset f_1\subset\ldots \subset c_{n-1}\supset f_n)$ is a gallery for which $\tilde e_\alpha$ is defined. Let $m,j,k$ be as in \ref{not:operators} case (III), that is $j$ is minimal such that $f_j\subset H_{\alpha,m}$ and $k>j$ minimal such that $f_k$ is also contained in $H_{\alpha,m}$. We write $\gamma$ as a concatenation of the following  three sub-galleries: $\gamma_j=(f_0\subset c_1\ldots c_{j-1}\supset f_j)$, $\gamma_{jk}=(f_j\subset c_j \ldots  c_{k-1}\supset f_k) \text{ and } $ and $\gamma_k=(f_k\subset c_k \ldots c_{n-1}\supset f_n)$. That is $\gamma=\gamma_j\star\gamma_{jk}\star\gamma_k$. 
Then 
\[
\tilde e_\alpha(\gamma)= \gamma_j\star \rho_m^{op}(\rho_m^{-1}(\gamma_jk))\star\gamma_k.
\]
\end{thm}
\begin{proof}
Convince yourself that the sub-galleries $\gamma_j$ and $\gamma_k$ remain untouched under an application of $\tilde e_\alpha$ and that the gallery $\gamma_{jk}$, which lies in between $H_{\alpha,m}$ and $H_{\alpha,m-1}$ gets reflected along $H_{\alpha,m}$ by $e_\alpha$. It is then easy to see that the application of  $\rho_m^{op}\circ \rho_m^{-1}$ does just that.  
\end{proof}

\begin{figure}[h]
\begin{center}
\begin{overpic}[width=0.7\textwidth]{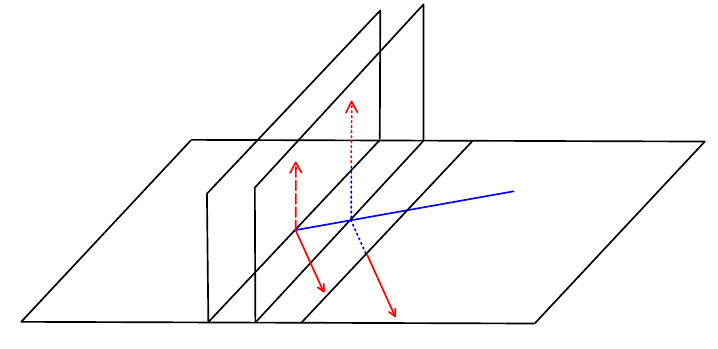}
\put(25,0){$H_{\alpha, m} $}
\put(33,0){$H_{\alpha, m+1} $}
\put(73,22){$\gamma = {\color{blue} \gamma_k^-} \star {\color{red} \gamma_k^+}$}
\put(40,7){\color{red} $\gamma_k^+$}
\put(60,17){\color{blue} $\gamma_k^-$}
\put(28,20){ $\rho_m^{-1}({\color{red}\gamma_k^+}$)}
\put(49,31){ $\iota_{m}\left( {\color{blue} \gamma_k^- }\star \rho_m^{-1}({\color{red}\gamma_k^+})\right)$}
\put(56,7){ $\rho_{A, \Cfm}\left( \iota_{m}\left({\color{blue} \gamma_k^- }\star\rho_m^{-1}({\color{red}\gamma_k^+})\right)\right) = e_\alpha(\gamma)$}
\end{overpic}
\caption[e-alpha]{Step by step illustration of the retractions we expressed the operator $e_\alpha$ with in Theorem~\ref{thm:ef}. Here $\iota_m=\iota_{m+1, m, B}$.}
\label{fig_operators-retractions}
\end{center}
\end{figure}
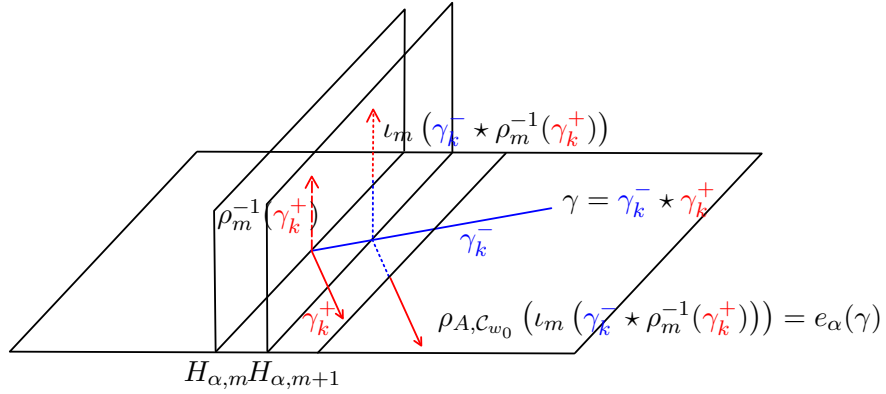

\begin{example}
In Figure~\ref{fig_operators-retractions} we illustrate the statement of Theorem~\ref{thm:ef} and show how the retractions are used to express the operator $e_\alpha$ with. We write $\iota_m$ for $\iota_{m+1, m, B}$. 

Let $\gamma$ be the concatenation of the bold blue gallery $\gamma_k^-$ and the bold red gallery $\gamma_k^+$. In the first step we keep the initial blue part and take a preimage of the red  piece under the retraction $\rho_m$. The second step consists of an application of a the map $\iota_m=\iota_{m+1, m, B}$ to the concatenation  ${\color{blue} \gamma_k^- }\star \rho_m^{-1}({\color{red}\gamma_k^+})$. This yields a gallery which coincides with $\gamma$ up to the wall $H_{\alpha, m+1}$. A final application of the opposite retraction, namely $\rho_{A, \Cfm}$, gives us the image of $\gamma$ under the root operator $e_\alpha$.   
\end{example}

\subsection{Interpretation in terms of root groups}

We use the fact that one can write pre-images of retractions in terms of groups to also rewrite the statements of Theorem~\ref{thm:ef} and \ref{thm:etilde}.  


\begin{corollary}
Let $\gamma=(f_0\subset c_1\supset f_1\subset\ldots \subset c_n\supset f_n)$  be a combinatorial gallery and let $k\geq 1$, $m\leq -1$ and $j$ be as in \ref{not:operators}. 

Suppose we are in case (I), that is $e_\alpha$ is defined for $\gamma$ and $\gamma$ is decomposed as $\gamma=\gamma_k^-\star\gamma_k^+$. Then there exist elements 
\[
u_1\in U^\ast_{-\alpha,-m}\setminus U^\ast_{-\alpha,-m-1},
\]
\[
u_2\in U^\ast_{-\alpha,-(m+1)}\setminus U^\ast_{\alpha,-m-2} \text{ and }
\]
\[ 
v \in U^\ast_{\alpha, -(m+1)}\setminus U^\ast_{\alpha,-m-2}
\]
 such that 
\[
e_\alpha(\gamma)=u_2\left( v\left(\gamma_k^-\star u_1(\gamma_k^+)\right)\right).
\]
With $u_1$ and $u_2$ as above and $m(u_i)$ as defined in \ref{def:m(u)}, one has
\[
e_\alpha(\gamma)=m(u_2)\left(\gamma_k^-\star m(u_1)(\gamma_k^+)\right). 
\]

Suppose we are in case (II), that is $f_\alpha$ is defined for $\gamma$  and $\gamma$ is decomposed as $\gamma=\gamma_j^-\star\gamma_j^+$. 
Then there exist elements 
\[
u_1\in U^\ast_{-\alpha,-(m+1)}\setminus U^\ast_{\alpha,-m-2} 
\]
\[
u_2\in U^\ast_{-\alpha,-m}\setminus U^\ast_{-\alpha,-m-1}\text{ and }
\]
 such that 
\[
f_\alpha(\gamma)=m(u_2)\left( \gamma_j^-\star m(u_1)(\gamma_j^+)\right).
\]

Suppose we are in case (III), that is $\tilde e_\alpha$ is defined. Decompose  $\gamma$ as a concatenation $\gamma_j\star\gamma_{jk}\star\gamma_k$ as in \ref{thm:etilde}. With $u\in U^\ast_{\alpha, m}\setminus U_{\alpha, m-1}$ and $m(u)$ as in \ref{def:m(u)}, we can write 
\[
\tilde e_\alpha(\gamma)= \gamma_j\star m(u) (\gamma_{jk}))\star\gamma_k.
\]
\end{corollary}

\begin{proof}
To see the statement in case (I) convince yourself that both $\rho_k$ and $u_1$ stabilize the half-apartment $H_{-\alpha, m}^+$ pointwise. Similarly $u_2$ and the retraction based at $\Cfm$ do fix $H_{-\alpha, m+1}^+$ and the half-apartment $H_{\alpha, m+1}^+$ is fixed by $v$ and the map $\iota_{m+1,m,B}$. With this we can read off the statement of Theorem~\ref{thm:ef} and the definition of the groups $U_{\beta,k}$.    
The remaining statements are shown accordingly. 
\end{proof}

\subsection{Galleries versus paths}\label{sec:paths}

The original path model was not based on galleries but on paths in the building. And therefore the root operators where also first defined on paths,  see \cite{LittelmannPaths}. When applying sequences of the path version of the root operators to a geodesic from the origin to a dominant weight, the resulting set of (folded) paths is such that each of the occuring paths is the image under a retraction of a pre-image of the geodesic one has started with under a different retraction. 
This is similar to the results on Hecke paths shown by Kapovich--Millson in \cite{KapovichMillson}. 
The crucial observation is the following lemma a proof of which was e.g. given in \cite[Lemma 3.1]{Marquis}.

Define the \emph{(simplicial) support} of a point $x\in X$  to be the smallest simplex containing $x$ in its interior.

\begin{lemma}\label{lem:properties}
Suppose $\pi:\left[0,1\right]\to X$ is a geodesic in a building.  Then $\pi$ is contained in a (not necessarily unique) minimal combinatorial gallery connecting the support of $\pi(0)$ with the support of $\pi(1)$. 
\end{lemma}

An immediate consequence of the lemma is that every geodesic $\pi:\left[0,1\right]\to X$ in a building $X$ is contained in an apartment and contained in the convex hull of its endpoints. 

Lemma~\ref{lem:properties} allows us to consider the image of a geodesic under a sequence of gallery root operators as follows. The image $\pi([0,l])$ of the path $\pi$ is, by the lemma, a subset of a minimal gallery $\gamma$. The image of $\pi$ under a sequence of gallery root operators is then the image of the set $\pi([0,l])$ inside $\gamma$ under this sequence of operators applied to $\gamma$.  
A straight forward comparison of the effect of the path root operators on $\pi$ with the image under the gallery root operators implies the following proposition. 

\begin{prop}\label{prop:pathgallery}
Let $\pi:[0,l]\to A$ be a geodesic in $A$ for which a sequence $\sigma$ of path root operators is defined and let $\gamma$ be a minimal combinatorial gallery containing $\pi$ such that the same sequence of gallery root operators is defined. Then the image of $\pi$ under said sequence is the same as the image of $\pi$ under the sequence of the corresponding gallery root operators.
\end{prop}

\renewcommand{\refname}{Bibliography}
\bibliography{bibliography}
\bibliographystyle{alpha}

\end{document}